\documentclass[12pt,reqno]{amsart}
\usepackage[top=1in, bottom=1in, left=1in, right=1in]{geometry}
\usepackage{times, amsthm, amssymb, amsmath, amsfonts, amsthm, bm, graphicx, mathrsfs}
\usepackage[all]{xy}
\usepackage[usenames,dvipsnames]{color}
\usepackage[pagebackref=true,pdftex]{hyperref}
\usepackage{comment}

\newtheorem{theorem}{Theorem}[section]
\newtheorem{lemma}[theorem]{Lemma}

\theoremstyle{definition}
\newtheorem{example}[theorem]{Example}
\newtheorem{remark}[theorem]{Remark}

\newtheorem{definition}[theorem]{Definition}

        {\begin{list}
                {\noindent\makebox[0mm][r]{\arabic{enumi}.}}
                {\leftmargin=5.5ex \usecounter{enumi}}
        }
        {\end{list}}

        {\begin{list}
                {\noindent\makebox[0mm][r]{(\roman{enumi})}}
                {\leftmargin=5.5ex \usecounter{enumi}}
        }
        {\end{list}}

\def\<{\langle}
\def\>{\rangle}
\def\0{\mathbf{0}}

\def\CC{{\mathbb C}}

\def\EE{{\mathbb E}}
\def\cE{{\mathcal E}}
\def\cF{{\mathcal F}}

\def\cJ{{\mathcal J}}

\def\NN{{\mathbb N}}

\def\QQ{{\mathbb Q}}

\def\RR{{\mathbb R}}

\def\ZZ{{\mathbb Z}}

\def\bbE{{\mathbb E}}

\def\del{\partial}

\def\boldzero{\boldsymbol{0}}

\def\GL{{\rm GL}}

\def\vol{{\rm \operatorname{vol}}}
\def\conv{{\rm \operatorname{conv}}}
\def\rank{{\operatorname{rank}}}

\def\codim{{\rm codim}}

\def\minus{\smallsetminus}
\def\nothing{\varnothing}

\newcommand*{\defeq}{\mathrel{\vcenter{\baselineskip0.5ex \lineskiplimit0pt
                     \hbox{\scriptsize.}\hbox{\scriptsize.}}}%
                     =}

\numberwithin{equation}{section}
\begin{document}

\mbox{}
\title[]{A sharp bound for hypergeometric rank in dimension three} 
 
\author{Christine Berkesch}
\address{School of Mathematics \\
University of Minnesota.}
\email{cberkesc@umn.edu}

\author{Mar\'ia-Cruz Fern\'andez-Fern\'andez}
\address{Departamento de \'Algebra \\
Universidad de Sevilla.}
\email{mcferfer@algebra.us.es}

\thanks{
CB was partially supported by NSF Grant DMS 2001101. 
\\
\indent \
MCFF was partially supported by projects PID2020-117843GB-I00 (Ministerio de Ciencia e Innovaci\'on, Spain), P20-01056 and US-1262169 (Consejer\'ia de Econom\'ia, Conocimiento, Empresas y Universidad, Junta de Andaluc\'ia and FEDER)}

\subjclass[2010]{13N10, 32C38, 33C70, 14M25.}

\begin{abstract}
We provide a sharp upper bound on the quotient of the rank of an $A$-hypergeometric system with a three-dimensional torus action by the normalized volume of $A$; in this case, the upper bound is two. 
\end{abstract}
\maketitle

\mbox{}
\vspace{-15mm}
\parskip=0ex
\parindent2em
\parskip=1ex
\parindent0pt

\setcounter{section}{1}
\section*{Introduction}
\vspace{-2mm}

$A$-hypergeometric $D$-modules, also known as GKZ-systems, were introduced in~\cite{GGZ,GKZ} to  generalize classical hypergeometric equations.
These systems of linear partial differential equations in several complex variables are determined by a matrix $A = (a_{i,j})\in\ZZ^{d\times n}$ with columns $a_k\in\ZZ^d$ such that $\ZZ A=\ZZ a_1 + \cdots +\ZZ a_n=\ZZ^d$ and a parameter vector $\beta\in\CC^d$.
We assume that $A$ is \emph{pointed}, i.e., that all columns in $A$ lie in a single linear halfspace of $\RR^d$ that does not contain the origin.

\begin{definition}
\label{def:hypgeom}
Let $x_1,x_2,\dots,x_n$ be coordinates on $\CC^n$, with corresponding partial derivatives $\del_1,\del_2,\dots,\del_n$, so that the Weyl algebra $D$ on $\CC^n$ is generated by $x_1,\dots,x_n,\del_1,\dots,\del_n$. 
Let 
\[
I_A \defeq \<\del^u-\del^v\mid  u,v\in\NN^n, Au=Av\> \subseteq \CC[\del_1,\ldots, \del_n]
\] 
denote the \emph{toric ideal} of $A$, and 
let $E_i\defeq\sum_{j=1}^n a_{i,j} x_j \del_j$ be the $i$th \emph{Euler operator} of $A$. 
The \emph{$A$-hypergeometric $D$-module} with parameter $\beta\in\CC^d$ is the left $D$-module 
\[
M_A(\beta)\defeq D/D\cdot\< I_A, E_1-\beta_1,\ldots, E_d-\beta_d\>.
\]
\end{definition}

For any choice of $A$ and $\beta$, the module $M_A(\beta)$ is \emph{holonomic} \cite{GGZ,adolphson}. 
Consequently, the dimension of the space of germs of holomorphic solutions of $M_A(\beta)$ at a nonsingular point, also known as its \emph{(holonomic) rank}, is finite. 
When $\beta\in\CC^d$ is generic, the rank of $M_A(\beta)$ is equal to the \emph{normalized volume} $\vol(A)$ of the matrix $A$~\cite{GKZ, adolphson} inside the lattice $\ZZ^d$, which is the Euclidean volume of the convex hull $\Delta_A\subseteq \RR^d$ of the columns of $A$ and the origin divided by $d!$; but in general, this is only a lower bound \cite{SST,MMW}.
The set 
\[
\cE(A)\defeq\{\beta\in\CC^d \mid \; \rank (M_A (\beta))>\vol(A)\} 
\]
is called the \emph{exceptional arrangement} of $A$, which is an affine subspace arrangement of codimension at least two that is closely related to the local cohomology modules of the toric ring $\CC[\del]/I_A$ \cite{MMW}. 
A parameter $\beta \in \cE(A)$ is called a \emph{rank jumping} parameter.
Combinatorial formulas to compute the rank of $M_A(\beta)$ in terms of the \emph{ranking lattices} $\EE^\beta$ of $A$ at $\beta$ appear in \cite{okuyama,berkesch}. 
Unfortunately, the presence of alternating signs in these formulas do not yield a strong upper bound for the rank of $M_A(\beta)$. 
One exception is the case $d=2$, where it was shown in~\cite{CDD} that $\rank(M_A(\beta))\leq\vol(A)+1$ is a sharp bound. 
For arbitrary $d$, $A$, and $\beta$, previously known upper bounds for the holonomic rank of $M_A(\beta)$ are as follows: 
\[
\rank (M_A (\beta))\leq 
	\begin{cases}
	4^d \cdot \vol(A) 
	& \text{if $I_A$ is homogeneous~\cite{SST},}\\
	4^{d+1}\cdot \vol(A) 
	& \text{otherwise~\cite{BFM-parametric}.}
	\end{cases}
\]
However, it is believed that these upper bounds are much too large. 
In~\cite{BerFer}, we showed that when $\beta$ is generic among rank-jumping parameters (i.e., \emph{simple} in~\cite{berkesch}), then 
\begin{equation}
\label{eq:simple-bound}
\dfrac{\rank(M_A(\beta))}{\vol(A)} \leq (d-1),
\end{equation}
and we showed that this bound is tight by constructing a sequence of examples for which the ratio $\rank(M_A(\beta))/\vol(A)$ tends to $d-1$. 
Still, this case neglects the rank-jumping parameters $\beta$ for which the highest jumps are possible. In fact, there are families of examples for which the ratio $\rank(M_A(\beta))/\vol(A)$ grow exponentially with $d$ \cite{Fer-exp-growth}. In this note, we show that when $d=3$, the bound~\eqref{eq:simple-bound} holds for all parameters $\beta$, with a strict inequality.

\begin{theorem}
\label{thm:d3bound}
There is a strict sharp inequality for all $A\in\ZZ^{3\times n}$ and $\beta\in\CC^3$: 
\begin{equation}
\label{eq:main-d3bound}
\dfrac{\rank(M_A(\beta))}{\vol(A)}< 2.
\end{equation}
\end{theorem}

\subsection*{Outline}
We begin with results in Ehrhart theory in~\S\ref{sec:ehrhart} and rank jumps in~\S\ref{sec:GKZ-prelims}. 
Next, in~\S\ref{sec:deg1}, we consider rank jumps in the case that $\Delta_A$ has degree one. 
Finally, the proof of Theorem~\ref{thm:d3bound} is completed in~\S\ref{sec:general-bound-dim-3}. 

\subsection*{Acknowledgements}
We thank Christian Haase for stating and proving Lemma~\ref{lem:Haase}, providing the catalyst for this article. 
We are also grateful to Laura Felicia Matusevich,  Vic Reiner, and Uli Walther for helpful conversations related to this work. 

\section{Preliminaries on Ehrhart theory}
\label{sec:ehrhart}

Let $\Delta\subseteq \RR^d$ be a lattice polytope, that is, the convex hull in $\RR^d$ of a finite set of points in $\ZZ^d$.
We denote by $|\Delta\cap\ZZ^d|$ the cardinality of the set of lattice points in $\Delta$.
The function $ g_\Delta (t): \NN \longrightarrow \NN$ defined by 
\[
g_\Delta (k):=|k\Delta \cap \ZZ^d| 
\] 
counts the number of lattice points in the $k$-fold dilatation of $\Delta$. 
Ehrhart proved that this function is a polynomial in $k$, which is now called the \emph{Ehrhart polynomial} of $\Delta$. Moreover, he also proved that when the polytope $\Delta\subseteq \RR^d$ is $d$-dimensional the degree of $g_\Delta (k)$ is $d$ and its leading coefficient is equal to the \emph{normalized volume} of $\Delta$ with respect to $\ZZ^d$, that is, the Euclidean  volume of $\Delta$ divided by $d!$~\cite{ehrhart}. 
The \emph{Ehrhart series} of $\Delta$ is the generating function 
\[
E_\Delta(t):=\sum_{k\geq 0} g_\Delta (k) t^k.
\]
The following result is well known (see, for example,~\cite{Batyrev-Nill} and the references therein).

\begin{theorem}
\label{thm:Ehrhart-series}
For a lattice polytope $\Delta\subseteq \RR^d$, there exists 
\[
h^*(t)=1+ h_1^* t + \cdots + h_k^* t^k,
\]
called the \emph{$h$-polynomial} of $\Delta$ 
such that the following properties hold:
\begin{enumerate}
\item $E_\Delta(t)=h^*(t)/(1-t)^{d+1}$,
\item $h^*_j\in\ZZ_{\geq 0}$ for all $j=1,\ldots, k$,
\item $\vol(\Delta)=1+h_1^*+\cdots + h_k^*$, 
\item $h_1^*=|\Delta\cap \ZZ^d|-d-1$, 
and 
\item the leading coefficient $h_k^* = |(d + 1 - k)\Delta^\circ \cap \ZZ^d|$, where $\Delta^\circ$ denotes the interior of $\Delta$.
\end{enumerate}
\end{theorem}

The \emph{degree} of $\Delta$, denoted $\deg (\Delta)$, is defined to be the degree of the $h$-polynomial of $\Delta$.
An interesting fact is that 
\[
\deg (\Delta)=\min \{j \in \NN \mid |i \Delta^\circ \cap \ZZ^d |=\nothing, \; \forall 1\leq i \leq d-j \}. 
\] 

It follows from Theorem~\ref{thm:Ehrhart-series} that lattice polytopes of degree zero are basic simplices, that is, lattice polytopes whose vertices form an affine lattice basis of $\ZZ^d$. 
Batyrev and Nill classified those polytopes having degree one~\cite[Theorem~2.5]{Batyrev-Nill}; they proved that $\deg(\Delta)\leq 1 $ if and only if $\Delta$ is an \emph{exceptional simplex} or a \emph{Lawrence prism}, which we now define. 

First, if $F$ is a face of $\Delta$ such that 
$|\Delta \cap \ZZ^d |-|F\cap \ZZ^d|-\codim(F)=0$, then we say that $\Delta$ is an \emph{iterated pyramid} over $F$.
An \emph{exceptional triangle} is a $2$-dimensional basic simplex multiplied by $2$.
An \emph{exceptional simplex} is a simplex that is the $(d-2)$-fold pyramid over an exceptional triangle.
In other words, 
$\Delta$ is the convex hull in $\RR^d$ of 
\[
e_0, e_0+2(e_1-e_0),e_0+2(e_2-e_0),e_3,\dots,e_d
\]
where $e_0,\ldots, e_d$ is some affine lattice basis of $\ZZ^d$.
Finally, a \emph{Lawrence prism} of heights $b_1,\ldots, b_d\geq 0$, denoted by $L(b_1,\ldots, b_s)$, is the convex hull in $\RR^d$ of
\[
e_0, e_0+b_1 (e_d-e_0), e_1, e_1+b_2(e_d-e_0), \ldots , e_{d-1}, e_{d-1}+b_d(e_d-e_0),
\]
where $e_0,\ldots, e_d$ is some affine lattice basis of $\ZZ^d$.

A Lawrence prism $L(b_1,\ldots,b_d)$ has degree one if $b_1+\cdots+b_d\geq 2$ \cite[Proposition 2.4]{Batyrev-Nill}. Otherwise, it is a basic simplex of degree zero.

\begin{remark}
\label{remark: Lawrence prism volume}
The normalized volume of the Lawrence prism $L(b_1,\ldots, b_d)$ is $b_1+\cdots+b_d$.
\end{remark}

\section{Preliminaries on rank jumps}
\label{sec:GKZ-prelims}

In this section, we return to the setting of a pointed matrix $A\in\ZZ^{d\times n}$ with $\ZZ A =\ZZ^d$. 
We will soon restrict to the case $d=3$, but it is not needed for the following definitions.
Recall that $\Delta_A \subseteq \RR^d$ denotes the convex hull of the columns of $A$ and the origin. The set of columns of $A$ will be also denoted by $A$. 
A submatrix $F$ of $A$ (or a subset of its set of columns) is called a \emph{face} of $A$, denoted $F\preceq A$, if $\Delta_F$ is a face of the polytope $\Delta_A\subseteq \RR^d$ and $A \cap \Delta_F =F$. In particular, the empty set and $A$ are faces of $A$. For a face $F\preceq A$, consider the union of the lattice translates
\begin{align*}
\EE_F^\beta\defeq 
\big[\ZZ^d\cap(\beta+\CC F) \big]\minus(\NN A+\ZZ F) = \bigsqcup_{b\in B_F^\beta} (b+\ZZ F),
\end{align*}
where $B_F^\beta \subseteq \ZZ^d$ is a set of lattice translate representatives. 
Since $|B^\beta_{F}|$ is the number of translates of $\ZZ F$ appearing in $\bbE_F^\beta$, it is by definition equal to the difference between $[\ZZ^d\cap \QQ F:\ZZ F]$ and the number of translates of $\ZZ F$ along $\beta+\CC F$ that are contained in $\NN A + \ZZ F$.

Given the set
$\cJ(\beta)\defeq\{(F,b)\mid F\preceq A,\, b\in B_F^\beta \}$, 
the \emph{ranking lattices} of $A$ at $\beta$ are defined to be 
\begin{align*}
\EE^\beta 
  \defeq \bigcup_{(F,b)\in \cJ(\beta)} (b+\ZZ F).
\end{align*}
Note that the ranking lattices of $A$ at $\beta$ are precisely the union of those sets $(b+\ZZ F)$ contained in $\ZZ^d \setminus \NN A$ such that $\beta\in (b+\CC F)$. This is closely related to the set of holes of the affine semigroup $\NN A$, namely the set $(\ZZ^d \cap \RR_{\geq 0}A)\setminus \NN A$. 

A rank jumping parameter $\beta$ is \emph{simple} (for a face $G\preceq A$) if the set of maximal pairs $(F,b)$ in $\cJ(\beta)$ with respect to inclusion on $b+\ZZ F$ all correspond to a unique face $G\preceq A$.

The main result in \cite{berkesch} states how the rank of $M_A(\beta)$ can be computed from the combinatorics of $\EE^\beta$ and $\Delta_A$. 
An explicit formula for the rank is given when the rank jumping parameter $\beta$ is \emph{simple} for a face $G\preceq A$ (see \cite{okuyama} for this particular case); in this case, 
\begin{equation}\label{eqn:formula-simple-rank-jump}
\rank (M_A (\beta))=\vol(A)+ |B_G^\beta|\cdot (\codim(G)-1)\cdot \vol_{\ZZ G}(G).
\end{equation} 

\begin{example}
\label{ex:2faces}
It is shown in~\cite[Example~6.21]{berkesch} that if $\beta\in\CC^d$ is such that the ranking lattices of $A$ at $\beta$ involve only two faces, $F_1$ and $F_2$, then 
the rank jump of $M$ at $\beta$ is
\begin{align}\label{24}
\rank(M_A(\beta)) - \vol(A) = 
    \sum_{i=1}^2 \left(
    |B_{F_i}^{\beta}| \cdot [\codim(F_i)-1] \cdot \vol(F_i)
    \right) + |B_G^{\beta}| \cdot C^{\beta} \cdot\vol(G),
\end{align}
where 
$G = F_1\cap F_2$ and 
the constant $C^{\beta}$ is given by
\begin{eqnarray*}
C^{\beta} = \binom{\codim(G)}{2} - \codim(G) + 1 
    - \binom{\codim(F_1)}{2} - \binom{\codim(F_2)}{2} +
    \binom{\codim(\CC F_1+\CC F_2)}{2}. 
\end{eqnarray*}
\end{example}

When $d=3$, Okuyama~\cite{okuyama} provided a formula for the rank of $M_A(\beta)$, as follows. 

\begin{theorem}\cite[Theorem~2.6]{okuyama}
\label{thm:d=3}
Let $d=3$ and $\beta\in\CC^3$. 
Define an equivalence relation on $\cJ(\beta)$ by 
\[
(F,b)\sim (G,c) 
\quad\text{if and only if}\quad
(b+\ZZ F)\cap (c+\ZZ G)\neq \varnothing.
\] 
Let $\cF_A(\beta)$ denote the set of equivalence classes of $\cJ(\beta)$ under $\sim$. 
For each $\Lambda\in\cF_A(\beta)$, let 
$\cJ_\Lambda^\beta$ denote the order complex on the face poset of faces $F$ with $(F,b)\in \Lambda$ for some $b\in\ZZ^3$. 
Then the rank jump of $A$ at $\beta\in\CC^d$ is given by 
\[
\rank(M_A(\beta)) - \vol(A) = 
\sum_{\Lambda\in \cF_A(\beta)} j_\Lambda(\beta),
\]
where 
\[
j_\Lambda(\beta) := 
\begin{cases}
\displaystyle\sum_{\text{rays } F\in\cJ_\Lambda^\beta} \left(\vol(F) - 1\right) + m-1
& \text{if $\widetilde{H}_{p}(\cJ_\Lambda^\beta)\cong 0$ for $p\neq 0$, $\widetilde{H}_{0}(\cJ_\Lambda^\beta)\cong \CC^{m-1}$ with $m>1$,}\\
\vol(F)
& \text{if $\widetilde{H}_{p}(\cJ_\Lambda^\beta)\cong 0$ for all $p$, so $\cJ_\Lambda^\beta$ consists of a ray $F$,}\\
2 
& \text{if $\widetilde{H}_{-1}(\cJ_\Lambda^\beta)\cong \CC$, $\widetilde{H}_{p}(\cJ_\Lambda^\beta)\cong 0$ for $p\neq -1$,}\\
0 
& \text{otherwise.}
\end{cases}
\]
\end{theorem}

\section{Rank jumps for $M_A(\beta)$ when $\Delta_A$ has degree one}
\label{sec:deg1}

Recall that we denote the convex hull of the columns of $A$ and the origin by $\Delta_A$. 
When $\Delta_A$ is a lattice polytope of degree one, it is an exceptional simplex or a Lawrence prism of heights $b_1,\ldots, b_d \in \NN$.
In this section, we compute the possible rank jumps for $M_A(\beta)$ when $\Delta_A$ is the former or if it is the latter and $d=3$.  

\begin{lemma}
\label{lem:exceptional-triangle}
If $\Delta=\Delta_A$ is an exceptional simplex, then $M_A(\beta)$ has no rank-jumping parameters. 
\end{lemma}
\begin{proof}
If $\Delta_A$ is an iterated pyramid over an exceptional triangle, then one of the following cases holds:
\begin{enumerate}
\vspace*{-2mm}
 \item[(i)] $\Delta_A$ is an exceptional triangle,
 \item[(ii)] $\Delta_A$ is (up to rigid transformation) the convex hull of the origin in $\RR^3$ and $\left[\begin{smallmatrix}
1&1&1\\0&2&0\\0&0&2\end{smallmatrix}\right]$,
\item[(iii)] $A$ is an \emph{iterated pyramid} as defined in \cite[Definition 3.4]{reducibility} over a face $F\preceq A$ for which $\Delta_F$ is a polytope of the form of $\Delta_A$ in cases (i) or (ii).
\end{enumerate}

For (iii), $\CC^d=\CC F \oplus \CC \overline{F}$ and $\rank (M_A(\beta))=\rank (M_F (\beta^F))$, where $\beta=\beta^F +\beta^{\overline{F}}$ for unique $\beta^F\in \CC F$ and $\beta^{\overline{F}}\in \CC \overline{F}$ (see \cite[Lemma 3.7]{reducibility}). Thus, it is enough to handle cases (i) and (ii).

For (i), we can assume for simplicity that $e_0$ is the origin of $\RR^2$. In this case, the vertices $2 e_1, 2 e_2$ of $\Delta_A$ are columns of $A$. Moreover, since $\ZZ A=\ZZ^d$, at least two elements of the set $\{ e_1, e_2, e_1+e_2\}$ are also columns of $A$.
If $e_1,e_2$ are columns of $A$, then $e_1+e_2$ is in $\NN A$ and $\NN A$ is normal. If $e_1+e_2$ and one $e_i$ are columns of $A$, then $\NN A$ has a one dimensional set of holes, given by $e_j+\NN\{e_j\}$ with $j\neq i$. As this is a codimension-one lattice translate in $\ZZ^2$ and it is the only set of holes in $\NN A$, it does not yield rank-jumping parameters by~\eqref{eqn:formula-simple-rank-jump}.

For (ii), since $\ZZ A = \ZZ^3$, it must be that at least two of the vectors in
$\left[\begin{smallmatrix}
	1&1&1\\1&0&1\\0&1&1 
\end{smallmatrix}\right]$
must also be in $A$. An exhaustive search reveals that none of these configurations yield any rank jumping parameters. 

Thus in all cases where $\Delta_A$ is the exceptional simplex, $M_A(\beta)$ admits no rank-jumping parameters~$\beta$. 
\end{proof}

The final case to consider in this section is that $\Delta_A$ is a Lawrence prism of heights $b_1,b_2, b_3 \in \NN$.

\begin{lemma}
\label{lem:lawrence}
If $\Delta$ is isomorphic to a three-dimensional Lawrence polytope of the form $L(b_1,b_2,b_3)$, then 
\[
\rank(M_A(\beta))<2\cdot \vol(A).
\] 
\end{lemma}
\begin{proof}
Since $M_A(\beta)$ is invariant under $\GL_3(\ZZ)$-transformation and its rank is invariant under reordering of the columns of $A$, we first note that if $\Delta = \Delta_A$ is isomorphic to a Lawrence prism, then we may assume for simplicity that $\Delta = L(b_1,b_2,b_3)$ as in \S\ref{sec:ehrhart}, where $e_0$ is the origin and $\{e_1, e_2, e_3\}$ is the standard basis for $\RR^3$.

If either $b_2$ or $b_3$ are zero, then $A$ is a pyramid over a face of dimension $2$ and this implies that $\rank (M_A(\beta))\leq \vol(A)+1$; indeed, the pyramid construction does not increase rank~\cite{reducibility}, and $\vol(A)+1$ is the maximal rank possible when $d=2$~\cite{CDD}. 
Thus, we can assume that $b_2, b_3\geq 1$. 

If $b_2,b_3\geq 1$ but $b_1=0$, then $\Delta$ has four edges that contain the origin, each of volume $1$ and lattice index $1$. 
Working from Theorem~\ref{thm:d=3}, there are six nontrivial potential order complexes $\cJ_\Lambda^\beta$ to consider, noting that each $\beta$ has a unique nontrivial $\Lambda\in\cF_A(\beta)$.
For the possible $\cJ_\Lambda^\beta$ with $\widetilde{H}_0(\cJ_\Lambda^\beta)\cong \CC^{m-1}$ for $m>1$, $j_A(\beta) = j_\Lambda(\beta) = m-1\leq 3$. 
In the remaining cases, $j_A(\beta) \in\{1,2\}$. 
It now follows that 
\begin{align}
\label{eq:b1=0:rankBound}
\rank(M_A(\beta))\leq 3+\vol(A). 
\end{align}
If $\vol(A)$ is $2$ or $3$, then $\Delta_A$ is equal to the convex hull of the origin and the columns of 
$\left[\begin{smallmatrix}
1 & 1 & 0 & 0\\
0 & 0 & 1 & 1\\
0 & 1 & 0 & 1
\end{smallmatrix}\right]$
or 
$\left[\begin{smallmatrix}
1 & 1 & 0 & 0\\
0 & 0 & 1 & 1\\
0 & 2 & 0 & 1
\end{smallmatrix}\right]$, 
respectively. In both cases, $\cE_A=\varnothing$, so $j_A(\beta)=0$.
Thus, if $b_1=0$, $b_2\geq 1$, $b_3\geq 1$, and $j_A(\beta)>0$, then $\vol(A)\geq 4$.
Therefore by~\eqref{eq:b1=0:rankBound}, 
$\rank(M_A(\beta))<2\cdot\vol(A)$ when $b_1=0$. 

We can assume for the rest of the proof that $b_1, b_2, b_3\geq 1$. Thus, the edges of $A$ are the lines $\ell_1=\RR e_1\cap A,
\ell_2=\RR e_2\cap A$, and $\ell_3=\RR e_3 \cap A$. 
However, since the vertices $e_1$ and $e_2$ of $\Delta$ are necessarily columns of $A$, 
$\vol_{\ZZ \ell_j }(\ell_j)=1$ for all $j=1,2$ and 
$|B_{\ell_j}^\beta| \leq 1$ for all $j=1,2$ and all $\beta\in \CC^3$. 

Again working from Theorem~\ref{thm:d=3}, there are seven possible order complexes $\cJ_\Lambda^\beta$ for a given $\Lambda\in\cF_A(\beta)$. 
We consider these options now.

Let $v$ denote the normalized volume of the convex hull of $A\cap \RR e_3$ with the origin inside the lattice spanned by $\ZZ (A\cap \RR e_3)$. 
If the maximal lattice translates under inclusion in $\Lambda$ are the three rays corresponding to $\ell_1$, $\ell_2$, and $\ell_3$, then $j_\Lambda(\beta) = 1+v$. 
If all maximal elements under inclusion in $\Lambda$ are facets, then Theorem~\ref{thm:d=3} implies that $j_\Lambda(\beta)=0$. 
The remaining cases involve maximal elements of $F_A(\beta)$ being a facet and a line, two lines, or one line. In each case, 
\[
j_\Lambda(\beta)= 
\begin{cases}
	v & \text{if (one of) the line(s) is $\ell_3$,}\\
	1 & \text{if $\ell_3$ is not among the lines.}
\end{cases}
\]
In all cases, the only additional $\Lambda$ available in $\cF_A(\beta)$ come from translates of $\ell_3$. Such $\Lambda$ will each have $j_\Lambda(\beta) = v$, and there are at most $[\ZZ e_3:\ZZ (A\cap \RR e_3)]-1$ such $\Lambda$ in $\cF_A(\beta)$.

Comparing the possible sums of $j_\Lambda(\beta)$ that are allowed in Theorem~\ref{thm:d=3} when computing the rank jump of $M_A(\beta)$ at $\beta$, it follows that 
\begin{align*}
\rank(M_A(\beta))-\vol(A) 
	&\leq 1+ [\ZZ e_3:\ZZ (A\cap \RR e_3)]\cdot v \\
	&= 1+ [\ZZ e_3:\ZZ (A\cap \RR e_3)]\cdot \vol_{\ZZ (A\cap \RR e_3)}(\conv(\{\boldzero, A\cap \RR e_3\}))\\
	&=1+b_3. 
\end{align*}
Finally, we rearrange the inequality to compute the desired result when $b_1+b_2\geq 2$, the last remaining case:
\begin{align*}
\rank(M_A(\beta)) 
	&\leq \vol(A)+1+b_3\\
	&< \vol(A) + 2+b_3\\
	&\leq \vol(A) +b_1+b_2+b_3\\
	&=2\cdot\vol(A). 
	\qedhere
\end{align*}
\end{proof}

\section{Upper bound for rank in dimension three}
\label{sec:general-bound-dim-3}

In this section, we prove Theorem~\ref{thm:d3bound}. First, we state a lemma shared with us by  Christian Haase.

\begin{lemma}
\label{lem:Haase}
Let $\Delta\subseteq \RR^3$ be a convex lattice polytope and $\ell_1,\ldots,\ell_r$ be the edges of $\Delta$ that contain a fixed vertex $v$. 
Then 
\begin{equation}
\label{eq:sum-vol-edges-bound}
\sum_{j=1}^r \vol(\ell_j)
\leq 
\vol(\Delta) + 2.
\end{equation}
Moreover, if equality holds in~\eqref{eq:sum-vol-edges-bound}, then $\Delta$ is a $3$-simplex with at least one facet of normalized volume one. 
\end{lemma}
\begin{proof}
Since $\vol (\ell)=|\ZZ^3 \cap \ell|-1$ for any edge $\ell$, 
then by Theorem~\ref{thm:Ehrhart-series}, 
\begin{equation}\label{eq:inequalities_bc}
\sum_{j=1}^r \vol(\ell_j) +1 
\leq 
|\Delta\cap\ZZ^3| 
= 
h_1^* +4 
\leq 
\vol(\Delta) + 3.
\end{equation}
This yields the desired inequality. 
Moreover, if
\begin{equation}
\label{eq:last-ineq}
h_1^* +4 
= 
\vol(\Delta) + 3,
\end{equation}
then $h_2^*=h_3^*=0$, 
so $\Delta$ has degree $1$. 
Batyrev and Nill characterized all polytopes with $h$-polynomial of degree one in~\cite[Theorem 2.5]{Batyrev-Nill}. 
Within this classification, the only polytopes for which the first inequality in~\eqref{eq:inequalities_bc} is an equality are precisely the simplices with at least one facet of volume one.
\end{proof}

\begin{proof}[Proof of Theorem~\ref{thm:d3bound}]
By~\cite[Corollary 2.2]{BerFer}, we may assume without loss of generality that $\beta\in \cE_A$ is not simple. 
From Okuyama's formula in Theorem~\ref{thm:d=3}, in the case $d=3$, 
\begin{equation}
\label{eq:Okuyama-formula-bound}
\rank(M_A(\beta))-\vol(A)
\leq 
\left(\sum_{F}\vol_{\ZZ^3 \cap \CC F}(F)\right)-1,
\end{equation} 
where $F$ runs over all one dimensional faces of $A$.

By way of contradiction, suppose that there is some $A \in \ZZ^{3\times n}$ and $\beta \in \CC^3$ such that 
\[
\rank (M_A(\beta))\geq 2 \cdot\vol (A);
\]
in particular, the rank jump of $M_A(\beta)$ at $\beta$ would be at least $\vol(A)$. Then by~\eqref{eq:Okuyama-formula-bound},
\begin{align}
\label{eq:sum-rays}
\vol(A) \leq 
\left(\sum_{F}\vol_{\ZZ^3 \cap \CC F}(F)\right)-1,
\end{align} 
where the summation runs over all edges $F$ in $\Delta$ that contain the origin. 
Combining this 
with~\eqref{eq:inequalities_bc} yields
\begin{equation}
\label{eq:inequalities_abc}
\vol(A) +2 
\leq 
\left(\sum_{F}\vol_{\ZZ^3 \cap \CC F}(F)\right)+1 
\leq 
|\Delta \cap\ZZ^3| 
= 
h_1^* +4 
\leq 
\vol(A) + 3,
\end{equation}
where again the summation runs over all edges $F$ in $\Delta$ that contain the origin.
Comparing the outer terms of~\eqref{eq:inequalities_abc}, it follows that exactly two of the three inequalities present must be equalities. 
We distinguish two cases, based on the third inequality. 

First, consider the case in which the third inequality in~\eqref{eq:inequalities_abc} is an equality, 
so that $h_1^*+1=\vol (A)$. 
This implies by Theorem~\ref{thm:Ehrhart-series} that $h_2^*=h_3^*=0$. 
Thus by~\cite[Theorem 2.5]{Batyrev-Nill}, $\Delta$ is either an (iterated) pyramid over the exceptional triangle or a Lawrence polytope. 
These cases are handled in Lemmas~\ref{lem:exceptional-triangle}~and~\ref{lem:lawrence}, showing that the inequality~\eqref{eq:main-d3bound} is strict in this case. 

Finally, we are left to consider the case that the third inequality in~\eqref{eq:inequalities_abc} is strict, so that the first two inequalities are both equalities. 
Now~\eqref{eq:inequalities_abc} becomes 
\[
\vol(\Delta) +2 
= 
\left(\sum_{F}\vol_{\ZZ^3 \cap \CC F}(F)\right)+1 
= 
|\Delta \cap\ZZ^3| 
= 
h_1^* + 4 
<
\vol(\Delta) + 3.
\]
Since the summation is over all edges $F$ of $\Delta$ with the origin as a vertex, the second equality implies that all lattice points in $\Delta$ lie on an edge of $\Delta$ that has the origin as a vertex and $\Delta$ has no interior lattice points. 
Thus every edge of $\Delta$ that does not contain the origin has volume $1$. 
We will show that no $\Delta$ fitting this case admits a rank-jump higher than one. 
To begin, note that $h_1^* + 2=\vol(\Delta)$, so $h_2^* + h_3^*=1$ since $h_1^*+h_2^*+h_3^*+1=\vol(\Delta)$. 

If $h_2^*=0$ and $h_3^*=1$, then 
$\Delta$ must have an interior lattice point by property (5) in Theorem~\ref{thm:Ehrhart-series},
a contradiction. 
Thus, we must be in the case that $h_2^*=1$ and $h_3^*=0$. 
Given that $d=3$, it follows that $\deg(\Delta)=2$. 
By \cite[Theorem~2]{treutlein}, 
$\Delta$ satisfies one of two possible cases.  

First, 
$\Delta$ could be isomorphic to 
the convex hull of the origin and the columns of 
$\conv\left[\begin{smallmatrix}
	3 & 0 & 0\\
	0 & 3 & 0\\
	0 & 0 & 1
\end{smallmatrix}\right]$ in $\RR^3$, 
so that $\vol(\Delta)=9$ and $|\Delta\cap\ZZ^3|=11$. 
In this case, $\Delta$ is a pyramid over a face of dimension two, so by~\cite{CDD,reducibility}, 
\[
\rank(M_A(\beta))\leq \vol (A) +1<2\cdot\vol(A). 
\]

Second, since $\Delta$ is not isomorphic to the convex hull of the origin and the columns of $\conv\left[\begin{smallmatrix}
	3 & 0 & 0\\
	0 & 3 & 0\\
	0 & 0 & 1
\end{smallmatrix}\right]$ in $\RR^3$, then~\cite[Theorem~2]{treutlein} implies that the following equivalent statements hold: 
\begin{enumerate}
\item $\vol(\Delta)\leq 4\cdot(|(2\Delta)^\circ\cap\ZZ^3|+1)$,
\item $|\Delta\cap\ZZ^3|\leq 3\cdot|(2\Delta)^\circ\cap\ZZ^3|+7$, and 
\item $|\Delta\cap\ZZ^3|\leq \frac{3}{4}\cdot\vol(\Delta)+4$. 
\end{enumerate}
Also, by~\cite[Lemma~9]{treutlein},  
$\vol(\Delta)=|\Delta\cap\ZZ^3|+|(2\Delta)^\circ\cap\ZZ^3|-3$. 
Combining this with 
$|\Delta\cap\ZZ^3|=\vol(\Delta)+2$ from~\eqref{eq:inequalities_abc} implies that $|(2\Delta)^\circ\cap\ZZ^3|=1$. 
Thus 
\[
|\Delta\cap\ZZ^3|\leq 3|(2\Delta)^\circ\cap\ZZ^3|+7=10
\quad\text{and}\quad 
\vol(\Delta)=|\Delta\cap\ZZ^3|-2\leq 8.
\]
Now the Ehrhart polynomial of $\Delta$,
\[
g_\Delta(t) = g_3t^3+g_2t^2+g_1t+1,
\]
satisfies the following constraints. 
First, $|\Delta\cap\ZZ^3|=\vol(\Delta)-2$, so 
\[
g_3+2 = \vol(\Delta)+2 = b = g_\Delta(1) = g_3+g_2+g_1+1,
\]
which implies that $g_1=1-g_2$. 
Next, since $\Delta$ has no interior lattice points, by Ehrhart reciprocity, which states that $g_{\Delta^\circ}(t) = (-1)^dg_\Delta(-t)$, 
\[
0 = |\Delta^\circ\cap\ZZ^3| 
  = -g_\Delta(-1) 
  = g_3-g_2+g_1-1.
\] 
Finally, since $i=|(2\Delta)^\circ\cap\ZZ^3|=1$, 
$1 = -g_{\Delta}(-2) 
  = 8g_3-4g_2+2g_1-1$.  
However, the equations 
\[
g_1=1-g_2,
\quad 
0 = g_3-g_2+g_1-1, 
\quad
\text{and} 
\quad 
1 = 8g_3-4g_2+2g_1-1 
\]
are incompatible, so there is no $\Delta$ that fits this case. 
Having exhausted all possibilities, we conclude that if the third inequality in~\eqref{eq:inequalities_abc} is strict, then $\rank(M_A(\beta))<2\cdot\vol(A)$.

Finally, the sequence of examples constructed in~\cite[Section 3]{BerFer} proves that the upper 
bound in~\eqref{eq:main-d3bound} is sharp.
\end{proof}

\raggedbottom
\def\cprime{$'$} \def\cprime{$'$}
\providecommand{\MR}{\relax\ifhmode\unskip\space\fi MR }
\providecommand{\MRhref}[2]{%
  \href{http://www.ams.org/mathscinet-getitem?mr=#1}{#2}
}
\providecommand{\href}[2]{#2}

\end{document}